\documentclass[12pt]{article}
\usepackage[a4paper,margin=2cm]{geometry}
\usepackage[USenglish]{babel}
\usepackage{latexsym,amsmath,enumerate,graphics,enumerate,amsthm,tikz,hyperref,float}
\usepackage[affil-it]{authblk}
\usepackage{enumerate,amsthm,dsfont,pstricks}
\usepackage{latexsym,amsmath,amssymb,amscd,wrapfig,graphicx}

\newtheorem{theorem}{Theorem}[section]
\newtheorem{lemma}[theorem]{Lemma}

\newtheorem{corollary}[theorem]{Corollary}

\newtheorem*{theorem*}{Theorem}

\theoremstyle{definition}
\newtheorem{definition}[theorem]{Definition}
\newtheorem{example}[theorem]{Example}

\newtheorem*{assumption*}{Assumption}
\newtheorem{remark}[theorem]{Remark}

\def\R{\mathbb{R}}

\def\eps{\varepsilon}
\def\0{\mathbf{0}}
\def\phi{\varphi}

\def\P{\mathcal{P}}

\def\ol{\overline}
\DeclareMathOperator{\diam}{diam}

\DeclareMathOperator{\Aut}{Aut}
\DeclareMathOperator{\Sp}{Sp}

\DeclareMathOperator*{\esssup}{ess\,sup}
\DeclareMathOperator{\ext}{ext}
\DeclareMathOperator{\co}{co}

\newcommand{\comment}[1]{}
\newcommand{\norm}[1]{\left\Vert #1 \right\Vert}

\numberwithin{equation}{section}

\makeatletter
\def\@maketitle{%
  \newpage
  \null
  \vskip 2em%
  \begin{center}%
  \let \footnote \thanks
    {\Large\bfseries \@title \par}%
    \vskip 1.5em%
    {\normalsize
      \lineskip .5em%
      \begin{tabular}[t]{c}%
        \@author
      \end{tabular}\par}%
    \vskip 1em%
    {\normalsize \@date}%
  \end{center}%
  \par
  \vskip 1.5em}
\makeatother

\begin{document}

\title{\sc \LARGE Hilbert isometries and maximal deviation preserving maps on JB-algebras}

\author{Mark Roelands%
\thanks{Email: \texttt{mark.roelands@gmail.com}}}
\affil{School of Mathematics, Statistics \& Actuarial Science, University of Kent, Canterbury, CT2 7NX,
United Kingdom}

\author{Marten Wortel%
\thanks{Email: \texttt{marten.wortel@up.ac.za}. This research was supported in part by the DST-NRF Centre of Excellence in Mathematical and Statistical Sciences (CoE-MaSS), grant reference number BA2017/260. }}
\affil{Department of Mathematics and Applied Mathematics, University of Pretoria, Private Bag X20 Hatfield, 0028 Pretoria, South Africa}

\maketitle
\date{}
%\date{\today}

\begin{abstract}
In this paper we characterize the surjective linear variation norm isometries on JB-algebras. Variation norm isometries are precisely the maps that preserve the maximal deviation, the quantum analogue of the standard deviation, which plays an important role in quantum statistics. Consequently, we characterize the Hilbert's metric isometries on cones in JB-algebras. 
\end{abstract}



\section{Introduction}

Linear bijections on JB-algebras that preserve the maximal deviation, which is the quantum analogue of the standard deviation, play an important role in quantum statistics, see \cite{Ha,M1,MB}. For a unital JB-algebra $A$ with state space $K$, the square of the maximal deviation is defined by $\sup_{\phi \in K} [\phi(x^2) - \phi(x)^2]$. In \cite{MB}, Moln\'ar and Barczy characterized these maps for the self-adjoint elements of the bounded operators on a Hilbert space. Later, in \cite{M1}, Moln\'ar generalized this characterization to von Neumann algebras without a type $I_2$ direct summand, which in turn was generalized by Hamhalter: in \cite[Theorem~1.1]{Ha} he showed that such a map $\Phi$ on a JBW-algebra without a type $I_2$ direct summand is of the form $\Phi(x) = \eps Jx + \phi(x)e$, where $\eps \in \{\pm 1\}$, $J$ is a Jordan isomorphism, $\phi$ is a linear functional, and $e$ is the unit. 

In the results mentioned above, a key observation is that the maximal deviation is one half of the diameter of the spectrum, which is a seminorm (with kernel the span of $e$). We call the induced norm on the quotient the variation norm, and so each maximal deviation preserving linear bijection induces a linear variation norm isometry on this quotient. The main result of this paper is a characterization of these linear surjective variation norm isometries on quotients of unital JB-algebras: in Theorem~\ref{thm:char_var_isoms}, we show that if $A$ and $B$ are unital JB-algebras, then each linear surjective isometry $S \colon A / \Sp(e) \to B / \Sp(e)$ is of the form $S = \eps [J]$ for some $\eps \in \{\pm 1\}$ and a Jordan isomorphism $J \colon A \to B$. A simple argument (see Lemma~\ref{lem:ham_char}) shows that this characterization implies Hamhalter's characterization, and so our main result generalizes Hamhalter's result.

Moreover, an important consequence of this result (in fact, this was our original motivation) is the characterization of surjective Hilbert's metric isometries on cones of unital JB-algebras. In \cite[Theorem~2.17]{LRW2}, extending the approach in \cite{Bo}, we showed that if $A$ and $B$ are unital JB-algebras and $f \colon A^\circ_+ \to B^\circ_+$ is a unital surjective Hilbert isometry, then $S = \log \circ f \circ \exp$ is a linear variation norm isometry from $A / \Sp(e)$ to $B / \Sp(e)$. Combined with the above characterization, this easily yields Theorem~\ref{thm:char_Hilbert_isoms}: every surjective Hilbert's metric isometry $f \colon A^\circ_+ \to B^\circ_+$ is of the form $f(x) = U_y J(x^\eps)$ for some $y \in B^\circ_+$, where $U_y$ denotes the quadratic representation. This result generalizes a result by Moln\'ar, \cite[Theorem~2]{M2}, that characterizes Hilbert's metric isometries on positive definite operators on a complex Hilbert space of dimension at least three, and our result in \cite{LRW2} where we prove this for JBW-algebras. Moreover, this result also complements our earlier work \cite{LRW1}, where we characterized the Hilbert's metric isometries on cones in $C(K)$-spaces. Other works on Hilbert's metric isometries on finite dimensional cones include \cite{dlH,LW,MT,Sp}.

As the name suggests, Hilbert's metric goes back to Hilbert \cite{H}, who defined a distance $\delta_H$ between points on an open bounded convex set $\Omega$ in a finite dimensional real vector space by

\[
\delta_H(x,y):=\log\left(\frac{\norm{x'-y}\norm{y'-x}}{\norm{x'-x}\norm{y'-y}}\right),
\] 
where $x'$ and $y'$ are the points of intersection of the line through $x$ and $y$ with the boundary $\partial \Omega$ such that $x$ is between $x'$ and $y$, and $y$ is between $x$ and $y'$. These Hilbert's metric spaces $(\Omega,\delta_H)$ are Finsler manifolds and generalize Klein's model for the real hyperbolic space. Hilbert's metric spaces also play an important role in the solution of Hilbert's fourth problem, see \cite{AP}, and possess features of nonpositive curvature \cite{B,KN}. The geometry of Hilbert's metric spaces has been studied increasingly over the past few years and a wide selection of results and theory can be found in the survey \cite{PT}. There is a slightly more general version of Hilbert's metric $d_H$, which describes the distance between pairs of rays in the interior of a cone in an order unit space. The metric $d_H$ is given in terms of the partial ordering induced by the cone and was introduced by Birkhoff in \cite{Bi}. This version of Hilbert's metric has found numerous applications in the spectral theory of linear and nonlinear operators, ergodic theory, and fractal analysis, see \cite{ACS,CPR1,CPR2,LL1,LL2,L,Neeb,Nuss,Up} and the references therein. 

The theory of JB-algebras seems to be closely related to understanding Hilbert's metric isometries on the interior of cones in order unit spaces. Supporting evidence for this claim was provided by Walsh in \cite{W}, who showed amongst other things that for finite dimensional order unit spaces $A$ the Hilbert's metric isometry group on $A_+^\circ$ does not equal the group of projectivities if and only if $A$ is a Euclidean Jordan algebra that is neither $\R^2$ nor a spin factor, see \cite[Corollary~1.4]{W}. In this case, the group of projectivities has index 2 in the isometry group and the alternative isometries are obtained by adjoining the inversion map $x\in A_+^\circ\mapsto x^{-1}\in A_+^\circ$. To date it remains an open problem whether these results can be generalized to infinite dimensional order unit spaces, but for JB-algebras, Theorem~\ref{thm:isom_group} shows that there are also more Hilbert's metric isometries on the cone than projectivities if and only if the corresponding JB-algebra is neither $\R^2$ nor a spin factor.
\\

We will now explain the ideas behind the proof of our main result, Theorem~\ref{thm:char_var_isoms}: the characterization of variation norm isometries $S \colon A / \Sp(e) \to B / \Sp(e)$. To obtain the desired Jordan isomorphism, we produce a projection orthoisomorphism: a bijection on projections that preserves orthogonality in both directions. A generalization of Bunce and Wright \cite{BW} of Dye's Theorem \cite{Dye} to JBW-algebras without a type $I_2$ direct summand will then yield a Jordan isomorphism. To construct this projection orthoisomorphism, we consider the geometry of the closed linear span of the extreme points of the dual balls of $A / \Sp(e)$ and $B / \Sp(e)$. These spaces are the preduals of the quotient of the atomic parts of $A^{**}$ and $B^{**}$ by $\Sp(e)$, and they are preserved by $S^*$; we denote the induced map between these spaces by $S'$. It turns out that the maximal norm closed faces of the balls of these spaces of at most a certain diameter are of the form $F_p - F_{p^\perp}$ (where $F_p$ is the face supported by $p$, see Section~\ref{sec:3.1}) with $p$ or $p^\perp$ an atom in $A^{**}$ or $B^{**}$. The fact that these maximal norm closed faces of at most a certain diameter have to be preserved by $S'$ yields a map between the atoms of $A^{**}$ and $B^{**}$ which can be shown to extend to a projection orthoisomorphism between the atomic parts. Now \cite[Corollary~2]{BW} yields a Jordan isomorphism between the atomic parts of $A^{**}$ and $B^{**}$ without their type $I_2$ direct summands, and for the type $I_2$ direct summand we construct a Jordan isomorphism using the map $S'^*$ and a characterization by Stacey \cite{S} of type $I_2$ JBW-algebras. In the final step we note that $A$ (resp.\ $B$) can be embedded into the atomic part of $A^{**}$ (resp.\ $B^{**}$), so the Jordan isomorphism restricts to a Jordan isomorphism from $A$ to $B$. 

\section{Preliminaries}\label{sec:2}

This section is concerned with providing some basic definitions, facts, and preliminary results about Hilbert's metric and JB-algebras.

\subsection{Order unit spaces}

Let $A$ be a partially ordered real vector space with cone $A_+$. So, $A_+$ is convex, $\lambda A_+\subseteq A_+$ for all $\lambda\ge 0$, $A_+\cap -A_+=\{0\}$, and the  partial ordering $\leq$ on $A$ is given by $x\leq y$ if $y-x\in A_+$. Suppose that there exists an {\em order unit} $u\in A_+$, i.e., for each $x\in A$ there exists $\lambda>0$ such that $-\lambda u\leq x\leq \lambda u$. Furthermore, assume that $A$ is {\em Archimedean}, that is to say, if $nx\leq u$ for all $n=1,2,\ldots$, then $x\leq 0$.  In that case $A$ can be equipped with the {\em order unit norm},
\[
\|x\|_u:=\inf\{\lambda>0 \colon -\lambda u\leq x\leq \lambda u\},
\]
and $(A,\|\cdot\|_u)$ is called an {\em order unit space}, see \cite{HO}.
It is not hard to show, see e.g. \cite{LRW1}, that $A_+$ has nonempty interior $A_+^\circ$ in $(A,\|\cdot\|_u)$ and $A_+^\circ=\{x\in A\colon \mbox{$x$ is an order unit of $A$}\}$.

On $A_+^\circ$  Hilbert's metric is defined as follows. For $x,y\in A_+^\circ$ let
\[
M(x/y):=\inf\{\beta>0\colon x\leq \beta y\}.
\]
Note that as $y\in A_+^\circ$ is an order unit, $M(x/y)<\infty$. On $A_+^\circ$, {\em Hilbert's metric} is given by
\begin{equation*}\label{d_H}
d_H(x,y) = \log M(x/y)M(y/x).
\end{equation*}
It is well known (cf.  \cite{LNBook,NMem}) that $d_H$ is a pseudo metric, as $d_H(\lambda x,\mu y)=d_H(x,y)$ for all $\lambda,\mu>0$ and $x,y\in A_+^\circ$. However, $d_H(x,y)=0$ for $x,y\in A_+^\circ$ if and only if $x=\lambda y$ for some $\lambda>0$, so  that $d_H$ is a metric on the set of  rays in $A_+^\circ$, which we shall denote by $\overline{A}_+^\circ$. Elements of $\overline{A}_+^\circ$ will be denoted by $\ol{x}$.

The set $\Aut(A_+)$ denotes the automorphisms of $A_+$, i.e., linear bijections of $A$ that preserve $A_+$. If $T \in \Aut(A_+)$, then $M(Tx/Ty) = M(x/y)$ and hence $T$ is a surjective Hilbert's metric isometry.

The \emph{state space} $K$ of an order unit space consists of all positive functionals mapping $u$ to one. By \cite[Theorem~1.19]{AS1}, the state space $K$ in the dual of an order unit space has the property that it generates the dual ball $B_{X^*}$ in the sense that it equals the convex hull
\begin{equation}\label{eq:ball_base_norm_space}
B_{X^*}=\mathrm{co}(K\cup-K).
\end{equation}
It follows that the dual cone $A_+^*$ of an order unit space is generating. Moreover, \cite[Proposition~1.26]{AS1} yields the following decomposition of elements in $A^*$.

\begin{lemma}\label{lem:decomposition_functionals}
Let $A$ be an order unit space. For any $\varphi\in A^*$ there are $\psi, \eta\in 
A_+^*$ such that $\varphi=\psi-\eta$ and $\|\varphi\|=\|\psi\|+\|\eta\|$.
\end{lemma}

\subsection{JB-algebras}

A \emph{Jordan algebra} $(A, \circ)$ is a commutative, not necessarily associative algebra such that
\[
x \circ (y \circ x^2) = (x \circ y) \circ x^2 \mbox{\quad  for all }x,y \in A.
\]
A \emph{JB-algebra} $A$ is a normed, complete real Jordan algebra satisfying,
\begin{align*}
\norm{x \circ y} &\leq \norm{x}\norm{y}, \\
\norm{x^2} &= \norm{x}^2, \\
\norm{x^2} &\leq \norm{x^2 + y^2}
\end{align*}
for all $x,y \in A$. If the JB-algebra has an algebraic unit $e$, it is an order unit space and the JB-norm mentioned above corresponds to the order unit norm $\|\cdot\|_e$. An element $x\in A$ is called \emph{central} if $x\circ(y\circ z)=y\circ(x\circ z)$ for all $y,z\in A$. The \emph{spectrum} of $x\in A$, denoted by $\sigma(x)$, is defined to be the set of $\lambda\in\R$ such that $x-\lambda e$ is not invertible in JB$(x,e)$, the JB-algebra generated by $x$ and $e$, see \cite[3.2.3]{HO}. Moreover, there is a functional calculus: JB$(x,e)\cong C(\sigma(x))$.

When studying Hilbert's metric on  $\ol{A}_+^\circ$ in unital JB-algebras, the  \emph{variation seminorm}
$\norm{\cdot}_v$ on $A$ given by,
\[
\norm{x}_v := \diam \sigma(x)=\max\sigma(x)-\min\sigma(x),
\]
will play an important role. The kernel of this seminorm is the span of $e$, and on the quotient space $[A]:= A / \Sp(e)$ it is a norm. We proceed to show that if
$\|\cdot\|_q$ is the quotient norm of $\|\cdot\|$ on $[A]$, then $2\|[x]\|_q =\|[x]\|_v$ for all $[x]\in [A]$. Indeed, for $[x]\in [A]$, using $\inf_{\lambda\in\R}\max\{t-\lambda,s+\lambda\}=(t+s)/2$, we have 
\begin{eqnarray}\label{eq:var_is_quotient}
2\|[x]\|_q & := & 2\inf_{\mu\in\mathbb{R}}\|x -\mu e\| \nonumber \\
 & = & 2\inf_{\mu\in\mathbb{R}} \max_{\lambda \in \sigma(x)}|\lambda -\mu| \nonumber \\
  & = & 2\inf_{\mu\in\mathbb{R}} \max\big\{{\textstyle\max_{\lambda \in \sigma(x)}}(\lambda -\mu), {\textstyle\max_{\lambda \in \sigma(x)}}(-\lambda +\mu)\big\}  \\
  & = &  \max\sigma(x) + \max-\sigma(x)=\max\sigma(x)-\min\sigma(x) \nonumber \\
& = & \|[x]\|_v. \nonumber
\end{eqnarray}

A JB-algebra $A$ induces an algebra structure on $\ol{A}_+^\circ$ by $\ol{x}\circ\ol{y}:=\ol{x\circ y}$, which is well-defined. We can also define $\ol{x}^{\alpha} := \ol{x^\alpha}$ for $\alpha\in\mathbb{R}$. Note that the map $\mathrm{log}\colon A_+^\circ\to A$ given by $x\mapsto \log(x)$ is a bijection, whose inverse $\mathrm{exp}$ is given by $x\mapsto \exp(x)$. Furthermore, as $\log(\lambda x ) =\log(x)+ \log(\lambda) e$ for all $x\in A_+^\circ$ and $\lambda>0$, the map $\mathrm{log}$ induces a bijection from $\ol{A}_+^\circ$ onto $[A]$ given by $\log\ol{x} = [\log x]$. Its inverse $\mathrm{exp}\colon [A]\to \ol{A}_+^\circ$ is given by $\exp([x]) = \ol{\exp(x)}$ for $[x]\in [A]$.

The \emph{Jordan triple product} $\{ \cdot, \cdot, \cdot \}$ is defined as
\[ \{x,y,z  \} := (x \circ y) \circ z + (z \circ y) \circ x - (x \circ z) \circ y, \]
for $x,y,z \in A$. For $y \in A$, the linear map $U_y \colon A \to A$ defined by $U_y x := \{y,x,y\}$ is called the \emph{quadratic representation} of $y$, and if $y \in A_+^\circ$, then $U_y \in \Aut(A_+)$. Note that $U_y$ induces a map $U_{\overline{y}} \colon \ol{A}_+^\circ \to \ol{A}_+^\circ$ for $y \in A_+^\circ$.

A \emph{JBW-algebra} is the Jordan analogue of a von Neumann algebra: it is a JB-algebra which is monotone complete and has a separating set of normal states, or equivalently, a JB-algebra that is a dual space. For a JBW-algebra $M$ we will denote its unique predual by $M_*$. If $A$ is a JB-algebra, one can extend the product in $A$ to $A^{**}$ turning $A^{**}$ into a JBW-algebra, see \cite[Corollary~2.50]{AS2}. In JBW-algebras the spectral theorem holds, which implies in particular that the linear span of projections is norm dense. If $p$ is a projection, then the orthogonal complement $e-p$ will be denoted by $p^\perp$. In the sequel we will denote the set of projections of a JBW-algebra by $\P(M)$. Every JBW-algebra decomposes into a direct sum of type I, II, and III JBW-algebras, and a JBW-algebra with trivial center is called a \emph{factor}. 

A minimal nonzero projection in a JB(W)-algebra is called an \emph{atom}. A JBW-algebra in which every nonzero projection dominates an atom is called an \emph{atomic} JBW-algebra. Every projection in an atomic JBW-algebra can be written as a supremum of atoms, and the supremum of all atoms yields a central projection. This central projection $z$ decomposes the JBW-algebra $M$ into a direct sum of subalgebras $M=zM\oplus z^\perp M$ where $zM$ is atomic and $z^\perp M$ is purely nonatomic. In the sequel we denote the atomic part $zM$ of a JBW-algebra $M$ by $M_a$. The composition of the canonical embedding $A \hookrightarrow A^{**}$ with the multiplication by $z$ is an isometric JB-algebra embedding of $A$ into $(A^{**})_a$. This is a standard result for $C$*-algebras, see e.g.\ \cite[Preliminaries]{A}, and the proof for JB-algebras is the same; see \cite[Proposition~1]{FR} for a proof for JB*-triples, which are a generalization of JB-algebras.

The next lemma characterizes a JB-algebra that is either $\R^2$ or a spin factor.

\begin{lemma}\label{lem:equivalences_R2_and_SF}
Let $A$ be a unital JB-algebra of dimension at least two and consider the following statements.
\begin{itemize}
\item[(1)] The inversion map $\iota\colon x\mapsto x^{-1}$ on $A_+^\circ$ is linear up to scalar multiplication.
\item[(2)] $\max_{x\in A}\#\sigma(x)=2$.
\item[(3)] $A$ is either $\R^2$ or a spin factor.
\item[(4)] There exists an atom $u\in A$ such that $u^\perp$ is also an atom.
\item[(5)] There are two atoms $u \not= v \in A$ such that $u \vee v = e$.
\end{itemize}

Then $(1)\Leftrightarrow(2)\Leftrightarrow(3)\Rightarrow(4) \Rightarrow (5)$. If $A$ is a JBW-algebra, then also $(5) \Rightarrow (4)\Rightarrow(3)$, so that they are all equivalent.
\end{lemma}

\begin{proof}
%$(1)\Leftrightarrow(2)$. Clearly, $\max_{x\in A}\#\sigma(x)= 2$ for all $x\in A$ when $A$ is either $\R^2$ or a spin factor.
$(1)\Rightarrow (2)$. If there is an $x\in A$ such that $\#\sigma(x)\ge 3$, then for JB$(x,e)\cong C(\sigma(x))$ the interior of the cone $C(\sigma(x))_+^\circ$ is invariant under $\iota$. Let $\lambda_1,\lambda_2,\lambda_3\in\sigma(x)$ be distinct. By Urysohn's lemma there is a function $f\in C(\sigma(x))_+^\circ$ such that $f(\lambda_i)=i$ for $i=1,2,3$. It follows that 
\[
\iota(f+\mathbf{1})=\alpha \iota(f)+\beta\iota(\mathbf{1})=\alpha f^{-1}+\beta\mathbf{1}\quad(\alpha,\beta>0)
\]
yields the linear equations $\alpha+\beta=\frac{1}{2}$, $\frac{1}{2}\alpha+\beta=\frac{1}{3}$, and $\frac{1}{3}\alpha+\beta=\frac{1}{4}$, which has no common solution. This implies that $\iota$ can not be linear up to scalar multiplication.

$(2)\Rightarrow (3)$. Suppose that $\max_{x\in A}\#\sigma(x)=2$. Let $y \in A$ be such that $\#\sigma(y)=2$, then $y = \lambda p + \mu p^\perp$ for some nontrivial projection $p$ and $\lambda \not= \mu$. Note that the spectral bound on the elements implies that $\dim(U_p(A)) = \dim(U_{p^\perp}(A)) = 1$, which shows in particular that $p$ and $p^\perp$ are atoms. Next, we will show that $p$ and $p^\perp$ are also atoms in $A^{**}$. Denoting the quadratic representation on $A^{**}$ by $V$, \cite[Proposition~2.4]{AS2} shows that $V_p$ is $w$*-continuous, and since $U_p^{**}$ is the unique $w$*-continuous extension of $U_p$, $U_p^{**}=V_p$. Since $U_p(A)$ is one-dimensional, so is $U_p^{**}(A^{**}) = V_p(A^{**})$ so $p$ is an atom in $A^{**}$. Similarly, $p^\perp$ is an atom in $A^{**}$. Now \cite[Lemma~5.53]{AS2} implies that $V_{p\lor p^\perp}(A^{**}) = A^{**}$ either equals $\R^2$ or a spin factor. Since $A^{**}$ is reflexive, $A$ is reflexive, and so $A$ is either $\R^2$ or a spin factor. 

$(3)\Rightarrow (1)$. If $A$ is $\R^2$, then $\iota(\lambda,\mu)=(\lambda \mu)^{-1}(\mu,\lambda)$ on $A_+^\circ$ which is linear up to scalar multiplication, and if $A$ is a spin factor, then $\iota(x,\lambda)=(\lambda^2-\langle x,x\rangle)^{-1}(-x,\lambda)$ on $A_+^\circ$ which is also linear up to scalar multiplication.

$(3)\Rightarrow(4)$. If $A$ is either $\R^2$ or a spin factor, then every nontrivial projection $u$ is an atom, hence so is $u^\perp$.

$(4) \Rightarrow (5)$. Trivial.

Suppose that $A$ is a JBW-algebra. To show $(5) \Rightarrow (4)$, \cite[Lemma~3.50]{AS2} states that either $v \leq u$ or that $u \lor v - u$ is an atom. The first is impossible since $u$ is an atom and $v \not= u$, and the second states that $u^\perp$ is an atom, as required.

$(4) \Rightarrow (3)$. If $u \in A$ is an atom such that $u^\perp$ is also an atom, then both $u$ and $u^\perp$ must be maximal as well. Therefore, the central cover $c(u)$ of $u$ satisfies $c(u) = u$ or $c(u) = e$. If $c(u)=u$, then
\[
A=U_u(A)\oplus U_{u^\perp}(A)=\R u \oplus \R u^\perp \cong \R^2
\] 
by \cite[Lemma~3.29]{AS2}. On the other hand if $c(u)=e$, then by the maximality of $u^\perp$ it follows that $c(u^\perp)=e$ as well, as otherwise $c(u^\perp)=u^\perp$ would imply that $u$ is a central projection. Hence, $A$ is of type $I_2$ by \cite[Definition~5.3.3]{HO}. If $z\in A$ is a central projection, then 
\[
z=U_u(z)+U_{u^\perp}(z)=\alpha u+\beta u^\perp=(\alpha-\beta)u+\beta e
\]
by \cite[Proposition~1.47]{AS2}. Since $u$ is not central, $\alpha=\beta$ and $z=\alpha e$, hence the center is trivial and so $A$ is a type $I_2$ factor, a spin factor. 
\end{proof}

\subsection{Orthogonality in the predual of a JBW-algebra}

For functionals in the predual $M_*$ of a JBW-algebra $M$ the decomposition from Lemma~\ref{lem:decomposition_functionals} is unique by \cite[Proposition~2.58]{AS2}. We will use the notation $\varphi=\varphi^+-\varphi^-$ for this unique decomposition.

\begin{definition}
Let $M$ be a JBW-algebra and let $\varphi\in M_*$ be positive. Then the smallest projection $p\in M$ such that $\norm{\varphi}=\varphi(p)$ is the {\em support projection} of $\varphi$ and is denoted by $s(\varphi)$.
\end{definition}

We can derive the following properties for the support projections of positive functionals.

\begin{lemma}\label{lem:support_projections}
Let $M$ be a JBW-algebra and let $\varphi,\psi\in M_*$ be positive. Then $s(\varphi)\le s(\varphi+\psi)$ and $s(\varphi+\psi)=s(\varphi)\vee s(\psi)$.
\end{lemma}

\begin{proof}
The first statement follows from
\[
\norm{\varphi+\psi}=(\varphi+\psi)(s(\varphi+\psi))=\varphi(s(\varphi+\psi))+\psi(s(\varphi+\psi))
\le\varphi(e)+\psi(e)=\norm{\varphi}+\norm{\psi}=\norm{\varphi+\psi},
\]
since this implies that $\norm{\varphi}=\varphi(s(\varphi+\psi))$ and therefore $s(\varphi)\le s(\varphi+\psi)$. Note that we have also shown $s(\varphi)\vee s(\psi)\le s(\varphi+\psi)$ and the second statement follows from
\[
(\varphi+\psi)(s(\varphi)\vee s(\psi))=\varphi(s(\varphi)\vee s(\psi))+\psi(s(\varphi)\vee s(\psi))=\norm{\varphi}+\norm{\psi}=\norm{\varphi+\psi},
\]
as this implies $s(\varphi+\psi)\le s(\varphi)\vee s(\psi)$.
\end{proof}

\begin{definition}
Let $M$ be a JBW-algebra and let $\varphi,\psi\in M_*$ be positive. Then $\varphi$ and $\psi$ are {\em orthogonal} if $\norm{\varphi-\psi}=\norm{\varphi}+\norm{\psi}$. In this case we write $\varphi\perp\psi$.
\end{definition}

By \cite[Lemma 5.4]{AS2} two positive functionals $\varphi$ and $\psi$ are orthogonal if and only if $s(\varphi)$ and $s(\psi)$ are orthogonal.

\begin{lemma}\label{lem:orthogonal_functionals1}
Let $M$ be a JBW-algebra and let $\varphi,\psi,\rho\in M_*$ be positive. Then $\varphi\perp\psi$ and $\varphi\perp\rho$ if and only if $\varphi\perp(\psi+\rho)$.
\end{lemma}

\begin{proof}
Suppose $\varphi\perp\psi$ and $\varphi\perp\rho$. Since $s(\psi)\le s(\varphi)^\perp$ and $s(\rho)\le s(\varphi)^\perp$, we have 
\[
s(\psi+\rho)=s(\psi)\vee s(\rho)\le s(\varphi)^\perp 
\]
by Lemma~\ref{lem:support_projections}, so $s(\varphi)\perp s(\psi+\rho)$ and hence $\varphi\perp(\psi+\rho)$. Conversely, if $\varphi\perp(\psi+\rho)$, then 
\[s(\varphi)\le s(\psi+\rho)^\perp\le s(\psi)^\perp, s(\rho)^\perp
\] 
again by Lemma~\ref{lem:support_projections}, so $s(\varphi)\perp s(\psi)$ and $s(\varphi)\perp s(\rho)$. Hence $\varphi\perp\psi$ and $\varphi\perp\rho$.
\end{proof}

\subsection{Duality and extreme points}

For a Banach space $X$, a subspace $U\subseteq X^*$ is called {\em norming} if $\norm{x} = \sup_{x^* \in B_U}|x^*(x)|$ for all $x\in X$. So, in particular, the map $x\mapsto\varphi_x$ from $X\to U^*$ where $\varphi_x(x^*):=x^*(x)$ is an isometric embedding. Note that $U^* \cong X^{**} / U^\perp$ (where $U^\perp := \{x^{**} \in X^{**} \colon x^{**}(U) = 0 \}$), and under this identification, $x \mapsto \phi_x$ is the composition of the natural embedding $X \hookrightarrow X^{**}$ with the quotient map $X^{**} \to X^{**}/ U^\perp$.

If $\ext(B_{X^*}) \subseteq U$, then $\co \ext(B_{X^*}) \subseteq B_U$ and so $U$ is norming by the Krein-Milman theorem. Therefore $\ol{\Sp}(\ext(B_{X^*})$, the norm closed linear span of $\ext(B_{X^*})$, is a straightforward norming example and we will denote it in the sequel by $X'$. Hence there is a canonical isometric embedding $X\hookrightarrow X'^*$. If $X^*$ is strictly convex, then $X' = X^*$, but in general this is not the case, see Example~\ref{ex:C[0,1]}. For a bounded operator $T\colon X\to Y$ we define $T'\colon Y'\to X^*$ by $T'y'(x):=y'(Tx)$ (so $T' = T^*|_{Y'}$).

\begin{remark}
Note that it is not true in general that $T'$ maps $Y'$ into $X'$. For example, consider $T \colon X \to \R$ represented by a functional $x^*\in X^*\setminus X'$. Then $T^*1=x^*\notin X'$.   
\end{remark}

However, the following lemma shows that if $T$ is an isometric isomorphism, the adjoint does satisfy this property.

\begin{lemma}\label{lem:X'}
Let $X$ and $Y$ be Banach spaces. If $T\colon X\to Y$ is an isometric isomorphism, then $T'\colon Y'\to X'$ is an isometric isomorphism.
\end{lemma}

\begin{proof}
The operator $T'$ is an isometry as it is the restriction of the isometry $T^*$. Since $T^*$ is an isometric isomorphism, $T^*(\mathrm{ext}(B_{Y^*}))=\mathrm{ext}(B_{X^*})$ and so $T^*(Y')=X'$ and $T'$ is surjective.
\end{proof}

By the above lemma, every isometric isomorphism $T\colon X\to Y$ extends to an isometric isomorphism $T'^*\colon X'^*\to Y'^*$.

\begin{remark}
Let $U\subseteq X$ be a closed subspace. Then $(X/U)^*=U^\perp\subseteq X^*$, but it is not true in general that the inclusion $(X/U)'\subseteq X'$ holds. Indeed, let $x^*\in X^*\setminus X'$ and $U:=\ker x^*$. Then $(X/U)'= (\ker x^*)^\perp = \Sp(x^*)$ has trivial intersection with $X'$.    However, in the special case where $A$ is a unital JB-algebra and $U = \Sp(e)$, it will be shown in \eqref{e:dual_of_quotient_in_dual} that $[A] ' = (A/\Sp(e))' \subseteq A'$.
\end{remark}

For an order unit space $A$ we know that $A'$ is the closed linear span of the pure states by \eqref{eq:ball_base_norm_space}. Moreover, if $A$ is a JB-algebra then it turns out that $A'$ is the predual of the atomic part of $A^{**}$.

\begin{lemma}\label{lem:predual_of_atomic_part}
Let $M$ be a JBW-algebra with normal state space $K$. Then $\ol{\Sp}\,{\mathrm{ext}(K)}^*\cong M_a$.
\end{lemma}

\begin{proof}
If $Y$ is a subspace of a Banach space $X$, then $Y^*\cong X^*/Y^\perp$. In our case we have 
\[
\ol{\Sp}\,{\mathrm{ext}(K)}^*\cong M/\ol{\Sp}\,{\mathrm{ext}(K)}^\perp. 
\]
By \cite[Lemma~3.42]{AS2} let $z$ be the central projection in $M$ such that $zM=M_a$. Then for any pure state $\varphi\in\mathrm{ext}(K)$ we have $s(\varphi)\le z$, so $U_{s(\varphi)}U_z=U_{s(\varphi)}$ and hence $U^*_zU^*_{s(\varphi)}=U^*_{s(\varphi)}$. It follows that
\[
\varphi(zx)=U_z^*\varphi(x)=U_z^*U_{s(\varphi)}^*\varphi(x)=U^*_{s(\varphi)}\varphi(x)=\varphi(x)
\]
for any $x\in M$ and so 
\[
\norm{zx}=\sup_{\varphi\in K}|\varphi(zx)|=\sup_{\varphi\in\mathrm{ext}(K)}|\varphi(zx)|=\sup_{\varphi\in\mathrm{ext}(K)}|\varphi(x)|.
\]
Hence $x\in z^\perp M$ if and only if $\norm{zx}=0$, which holds if and only if $\varphi(x)=0$ for all $\varphi\in\mathrm{ext}(K)$, which in turn is equivalent to $x\in\ol{\Sp}\,{\mathrm{ext}(K)}^\perp$. So $\ol{\Sp}\,{\mathrm{ext}(K)}^\perp=z^\perp M$ and finally since $M=zM\oplus z^\perp M$, we have $\ol{\Sp}\,{\mathrm{ext}(K)}^*\cong M_a$.
\end{proof}

\begin{corollary}\label{cor:A'*_atomic_bidual}
If $A$ is a JB-algebra, then $A'^* \cong (A^{**})_a$.
\end{corollary}

Although not directly relevant for this paper, we can prove the same result for von Neumann algebras.

\begin{lemma}\label{lem:predual_atomic_von_Neumann_algebra}
Let $\mathcal{M}$ be a von Neumann algebra with normal state space $K$. Then $\overline{\Sp} \, \ext (B_{\mathcal{M}_*})^* \cong \mathcal{M}_a$.
\end{lemma}

\begin{proof}
Let $\phi \in B_{\mathcal{M}_*}$, with polar decomposition $\phi = v |\phi|$. Then \cite[Theorem~2.1]{AR} states that $\phi \in \ext (B_{\mathcal{M}_*})$ if and only if $|\phi| \in \ext K$. So if 
$\phi \in \ext (B_{\mathcal{M}_*})$, then for $x \in M$, letting $z$ be the central projection in $\mathcal{M}$ such that $z \mathcal{M} = \mathcal{M}_a$,
\begin{align*}
\phi(zx) &= v |\phi|(zx) = |\phi|(zxv) = U_z^* |\phi|(xv) = U_z^* U_{s(|\phi|)}^* |\phi|(xv) \\
&= U_{s(|\phi|)}^* |\phi|(xv) = |\phi|(xv) = v |\phi|(x) = \phi(x).
\end{align*}
The rest of the argument is the same as in Lemma~\ref{lem:predual_of_atomic_part}.
\end{proof}

We conclude that $\mathcal{A}'^* \cong (\mathcal{A}^{**})_a$ for a $C$*-algebra $\mathcal{A}$. 

\begin{example}\label{ex:C[0,1]}
Let $\Omega$ be a compact Hausdorff space and consider $A=C(\Omega)$. For $s \in \Omega$ we denote by $\delta_s \in M(\Omega)$ the dirac measure in $s$. Then
\[
A'=\left\{\sum_{n=1}^\infty\lambda_n\delta_{s_n}\colon (\lambda_n)_{n\ge 1}\in\ell^1,\ s_n\in \Omega \right\}\cong \ell^1(\Omega), \quad A'^*\cong\ell^\infty(\Omega). 
\] 
The canonical embedding of $C(\Omega)$ into $\ell^\infty(\Omega)$ coincides with the embedding $A\hookrightarrow A'^*$.
\end{example}

\section{Geometry of the dual ball}\label{sec:3}

Our goal is to characterize the linear surjective variation norm isometries $S\colon [A] \to [B]$ for unital JB-algebras $A$ and $B$. To exclude trivial cases, from now on we make the following assumption.
\begin{assumption*}
All JB-algebras $A$ satisfy $\dim(A) \geq 2$.
\end{assumption*}

In \eqref{eq:var_is_quotient} we showed that the variation norm equals the quotient norm up to a factor of two, so we can consider the operator $S'$ between $[B]'$ and $[A]'$. To determine $[A]'$, we have to find the extreme points of the ball $2B_{e_A^\perp}$ of $[A]^*\cong e_A^\perp\subseteq A^*$, where $e_A^\perp := \{ \phi \in A^* \colon \phi(e) = 0 \}$. Since the double dual of a JB-algebra is a JBW-algebra, we adopt the more general setting of a JBW-algebra $M$ and we consider $e_{\perp} := \{\phi \in M_* \colon \phi(e) = 0\} \subseteq M_*$.
 
\subsection{The facial structure of $2B_{e_\perp}$}\label{sec:3.1}

We proceed to characterize the norm closed faces of $2B_{e_\perp}$, and consequently the extreme points since these correspond to the singleton faces. Recall that a subset $F$ of a convex set $C$ is a \emph{face} if it is convex and satisfies the property that if $tx+(1-t)y\in F$ for some $x,y\in C$, then $x,y\in F$. We call a face \emph{proper} if it is nonempty and not equal to $C$. 

Let $M$ be a JBW-algebra with normal state space $K$. For every norm closed face $F$ of $K$ we also have a support projection
\[
s(F):=\bigvee\{s(\varphi)\colon \varphi\in F\}
\]
and by \cite[Theorem~5.32]{AS2}, $F$ is of the form $F=F_p:=\{\varphi\in K\colon\varphi(p)=1\}$ with $p=s(F)$. Conversely, such a set $F_p$ defines a norm closed face of $K$. 

Consider $e_\perp\subseteq M_*$ with scaled unit ball $2B_{e_\perp}$. We will exploit the facial structure of the normal state space $K$ to determine the facial structure of $2B_{e_\perp}$. Since a normal functional $\varphi\in M_*$ satisfies  $\norm{\varphi^+}-\norm{\varphi^-}=\varphi^+(e)-\varphi^-(e)$, it follows that $\varphi\in e_\perp$ if and only if $\norm{\varphi^+}=\norm{\varphi^-}$, and so 
\[
e_\perp=\{\varphi\in M_*\colon \norm{\varphi^+}=\norm{\varphi^-}\}.
\]
The following theorem characterizes the proper norm closed faces of $2B_{e_\perp}$. Note that if $F \subseteq 2B_{e_\perp}$ is a proper face, then every element of $F$ has norm $2$.\begin{theorem}\label{thm:faces}
Let $M$ be a JBW-algebra with normal state space $K$. Then $F$ is a proper face of $2B_{e^\perp}$ if and only if $F=F_+-F_-$ where 
\[
F_+:=\{\varphi\in K\colon\varphi-\psi\in F\}\quad \mbox{and}\quad F_-:=\{\psi\in K\colon\varphi-\psi\in F\} 
\]
are proper orthogonal faces of $K$. Moreover, $F$ is norm closed if and only if $F_+$ and $F_-$ are norm closed.
\end{theorem}

\begin{proof}
Let $F\subseteq 2B_{e_\perp}$ be a proper face. The inclusion $F\subseteq F_+-F_-$ is clear, so suppose $\varphi\in F_+$ and $\psi\in F_-$. Let $\rho\in F_-$ and $\tau\in F_+$ be such that $\varphi-\rho,\tau-\psi\in F$. Then by convexity $\frac{1}{2}(\varphi-\rho)+\frac{1}{2}(\tau-\psi)\in F$, so
\[
{\textstyle\frac{1}{2}}(\varphi-\psi)+{\textstyle\frac{1}{2}}(\tau-\rho)={\textstyle\frac{1}{2}}(\varphi-\rho)-{\textstyle\frac{1}{2}}(\tau-\psi)\in F.
\]
Since $\varphi-\psi,\tau-\rho\in 2B_{e_\perp}$, it follows that $\varphi-\psi,\tau-\rho\in F$ as it is a face, hence $F_+-F_-\subseteq F$.

To see that $F_+$ and $F_-$ are orthogonal, pick $\varphi\in F_+$ and $\psi\in F_-$. Then for $\varphi-\rho, \tau-\psi\in F$ we have
\[
{\textstyle\frac{1}{2}}(\varphi+\tau)-{\textstyle\frac{1}{2}}(\rho+\psi)={\textstyle\frac{1}{2}}(\varphi-\rho)+{\textstyle\frac{1}{2}}(\tau-\psi)\in F,
\]
so
\[
2=\norm{{\textstyle\frac{1}{2}}(\varphi+\tau)-{\textstyle\frac{1}{2}}(\rho+\psi)}\le{\textstyle{\frac{1}{2}}}
\norm{\varphi+\tau}+{\textstyle\frac{1}{2}}\norm{\rho+\psi}=2
\]
and therefore $(\varphi+\tau)\perp(\rho+\psi)$. Applying Lemma~\ref{lem:orthogonal_functionals1} twice yields $\varphi\perp\psi$. Also note that this implies  $F_+$ and $F_-$ are nonempty proper subsets of $K$. 

To see that $F_+$ and $F_-$ are faces of $K$, suppose $\varphi,\psi\in F_+$ and $0<t<1$. If $\rho\in F_-$, then 
\[
t\phi+(1-t)\psi-\rho=t(\phi-\rho)+(1-t)(\psi-\rho)\in F
\]
by convexity of $F$, so $t\phi+(1-t)\psi\in F_+$ showing the convexity of $F_+$. Similarly, $F_-$ is convex. Suppose $\phi,\psi\in K$ with $t\phi+(1-t)\psi\in F_+$. If $\rho\in F_-$, then 
\[
t(\phi-\rho)+(1-t)(\psi-\rho)=t\phi+(1-t)\psi-\rho\in F.
\]
Hence $\phi-\rho,\psi-\rho\in F$ since $F$ is a face and so $\phi,\psi\in F_+$, showing that $F_+$ is a face of $K$. Similarly, $F_-$ is a face of $K$. 

Conversely, suppose $F_+$ and $F_-$ are orthogonal proper faces of $K$. If $\varphi-\psi,\rho-\tau\in F$ and $0\le t\le 1$, then
\[
t(\varphi-\psi)+(1-t)(\rho-\tau)=(t\varphi+(1-t)\rho)-(t\psi+(1-t)\tau)\in F_+-F_-=F,
\]
so $F$ is convex. Furthermore, if $\varphi-\psi\in F_+ - F_- = F$ is such that $\varphi-\psi=t\sigma+(1-t)\xi$ for some $\sigma,\xi\in 2B_{e_\perp}$, then we have
\begin{align*}
\norm{(t\sigma^++(1-t)\xi^+)-(t\sigma^-+(1-t)\xi^-)}&=\norm{\varphi-\psi}=\norm{\varphi}+\norm{\psi}=2\\&=
\norm{t\sigma^++(1-t)\xi^+}+\norm{t\sigma^-+(1-t)\xi^-}.
\end{align*}
Hence $(t\sigma^++(1-t)\xi^+) \perp (t\sigma^-+(1-t)\xi^-)$, so $\varphi=t\sigma^++(1-t)\xi^+$ and $\psi=t\sigma^-+(1-t)\xi^-$ by the uniqueness of the orthogonal decomposition of $\varphi-\psi$. It follows that $\sigma^+,\xi^+\in F_+$ and $\sigma^-,\xi^-\in F_-$, so $\sigma,\xi\in F$ making it a face of $2B_{e_\perp}$. Since $F_+$ and $F_-$ are proper, $F$ is proper.

As for the closedness statement, suppose $F_+$ and $F_-$ are norm closed and let $(\varphi_n-\psi_n)_{n\ge 1}$ be a sequence in $F$ such that $\varphi_n-\psi_n\to\rho\in 2B_{e_\perp}$. Let $p := s(F_+)$ and $q := s(F_-)$, so that $F_+=F_p$ and $F_-=F_q$. It follows directly from the orthogonality of $F_+$ and $F_-$ that $p\perp q$. Moreover, we have $U^*_p\varphi_n=\varphi_n$ and $U^*_p\psi_n=0$ for all $n\ge 1$ by \cite[Proposition~1.41]{AS2}, hence 
 \[
\norm{\varphi_n-U^*_p(\rho)}=\norm{U^*_p(\varphi_n-\psi_n-\rho)}\le\norm{U^*_p}\norm{\varphi_n-\psi_n-\rho} \to 0
\]
and so $\varphi_n\to U^*_p\rho\in F_+$ as $F_+$ is closed. Similarly, $\psi_n\to U^*_q\rho\in F_-$ and $\rho = \rho_+ - \rho_- = U^*_p\rho-U^*_q\rho\in F$ which shows that $F$ is closed.

Conversely, suppose $F$ is norm closed. Let $(\varphi_n)_{n\ge 1}\subseteq F_+$ and $(\psi_n)_{n\ge 1}\subseteq F_-$ be such that $\varphi_n\to\varphi$ and $\psi_n\to\psi$ for some $\varphi,\psi\in K$. Then for any $\rho\in F_-$ we have $\varphi_n-\rho\to \varphi-\rho\in F$, so $\varphi\in F_+$. Similarly, we find that $\psi\in F_-$ and so $F_+$ and $F_-$ are norm closed faces of $K$.
\end{proof}

It follows from Theorem~\ref{thm:faces} that all the proper norm closed faces of $2B_{e_\perp}$ are of the form $F_p-F_q$ for nontrivial orthogonal projections $p$ and $q$. Recall that the extreme points of a convex set are precisely the singleton faces, so the following corollary, which generalizes our previous result in \cite[Proposition~4.1]{LRW1}, characterizes the extreme points of $2B_{e_\perp}$.

\begin{corollary}\label{cor:extreme_points_dual_JB}
Let $M$ be a JBW-algebra with normal state space $K$. Then
\[
\mathrm{ext}(2B_{e_\perp})=\left\{\varphi-\psi\in e_\perp\colon\mbox{$\varphi,\psi\in K$ are orthogonal pure states}\right\}.
\]
\end{corollary}

If $A$ is a JB-algebra, then it follows from Corollary~\ref{cor:extreme_points_dual_JB} that 
\begin{equation}\label{e:dual_of_quotient_in_dual}
[A]' = \ol{\Sp}\left\{\varphi-\psi\in e^\perp \colon \varphi,\psi\in K\ \mbox{are orthogonal pure states}\right\}\subseteq\ol{\Sp}\, \mathrm{ext}(K) = A'.
\end{equation}

The following lemma shows that we can identify $[A]'$ with the annihilator of the unit $e^a$ of the atomic part $(A^{**})_a$ of $A^{**}$.

\begin{lemma}\label{lem:(A/e)'=e_perp}
Let $A$ be a unital JB-algebra with unit $e$ and state space $K$ and let $e^a \in (A^{**})_a$ denote the unit in the atomic part of $A^{**}$. Then 
\[
[A]' = e^a_\perp \subseteq ((A^{**})_a)_*.
\]
\end{lemma} 

\begin{proof}
Consider the dual pair $(A, A')$, where $A ' = \ol{\Sp}\, \ext K \subseteq ((A^{**})_a)_*$ (see Lemma~\ref{lem:predual_of_atomic_part}). Let $e \in A$ be the unit of $A$. We will first show that 
\begin{equation}\label{e:dual_e_perp}
e^\perp = \ol{\Sp}\left\{\varphi-\psi \colon\varphi,\psi\in K\ \mbox{are orthogonal pure states}\right\}.
\end{equation}
Clearly $\varphi - \psi \in e^\perp$ for orthogonal $\varphi, \psi \in K$, and then so is its closed linear span since $e^\perp$ is a closed linear subspace. Conversely, suppose $\varphi - \psi \in e^\perp$ with $\varphi \perp \psi$. By \cite[Theorem~5.61]{AS2} we can write 
\[
\varphi = \sum_{i=1}^\infty \lambda_i \varphi_i \quad \mbox{and} \quad \psi = \sum_{i=1}^\infty \mu_i \psi_i
\]
with $\lambda_i, \mu_i > 0$ and all the $\varphi_i$, as well as all the $\psi_i$, orthogonal pure states. Therefore $\varphi - \psi = \sum_{i=1}^\infty (\lambda_i \varphi_i - \mu_i \psi_i)$, and $\sum_{i=1}^\infty \lambda_i = \sum_{i=1}^\infty \mu_i$ since $\varphi - \psi \in e^\perp$. We claim that $\sum_{i=1}^\infty (\lambda_i \varphi_i - \mu_i \psi_i)$ can be written as $\sum_{k=1}^\infty \gamma_k(\phi_{i_k} - \psi_{i_k})$. If the claim holds, then
\[
\varphi - \psi = \sum_{k=1}^\infty \gamma_k(\phi_{i_k} - \psi_{i_k}) \in \ol{\Sp}\left\{\varphi-\psi \colon\varphi,\psi\in K\ \mbox{are orthogonal pure states}\right\}
\]
which shows \eqref{e:dual_e_perp}. Then $[A]'^* \cong (e^\perp)^* \cong (\ol{\Sp} \ext K)^* / e^{\perp \perp}$, and since $e^\perp$ has codimension one, $e^{\perp \perp}$ has dimension one and contains $e^a$, thus $[A]'^* \cong [(A^{**})_a]$ as required. 

It remains to prove the claim. Let $C$ be the set of elements of the form  $\sum_{i=1}^\infty (\alpha_i \phi_i - \beta_i \psi_i)$, where $\alpha_i, \beta_i \geq 0$ for all $i$ with $\sum_{i=1}^\infty \alpha_i = \sum_{i=1}^\infty \beta_i$. Note that $\xi_1 := \phi - \psi \in C$. Given a nonzero element $\xi_n = \sum_{i=1}^\infty (\alpha_i \phi_i - \beta_i \psi_i) \in C$, let $i_n := \min\{i \colon \alpha_i > 0\}$ and $j_n := \min\{j \colon \beta_j > 0\}$. Define $\gamma_n := \alpha_{i_n} \wedge \beta_{j_n}$ and $\xi_{n+1} := \xi_n - \gamma_n(\phi_{i_n} - \psi_{j_n}) \in C$. This process generates a sequence $(\xi_n)_{n \geq 1}$ in $C$ such that $\xi_n = \xi_1 - \sum_{k=1}^{n-1} \gamma_k(\phi_{i_k} - \psi_{i_k})$. Since $\xi_n \to 0$, we have that $\xi_1 =  \sum_{k=1}^\infty \gamma_k(\phi_{i_k} - \psi_{i_k})$ as required.
\end{proof}

Consequently, $[A]'^* \cong [(A^{**})_a]$. (Note that the proof of Lemma~\ref{lem:(A/e)'=e_perp} also shows that $[\mathcal{A}]'^* \cong [(\mathcal{A}^{**})_a]$ for $C$*-algebras $\mathcal{A}$.)

We return to the situation where $S \colon [A] \to [B]$ is an isometric isomorphism. Let $e_A$ be the unit of $A$. Under the canonical embedding $A \hookrightarrow A'^* \cong (A^{**})_a$, $e_A$ corresponds to the unit of $(A^{**})_a$. Therefore, by Lemma~\ref{lem:(A/e)'=e_perp}, $S' \colon e_B^\perp \to e_A^\perp$ is an isometric isomorphism. Since $e_A^\perp$ and $e_B^\perp$ are the predual of the JBW-algebras $(A^{**})_a$ and $(B^{**})_a$, the facial structure of the balls $2B_{e_A^\perp}$ and $2B_{e_B^\perp}$ follows from Theorem~\ref{thm:faces}. Our next goal is to describe an important subclass of these faces.

\begin{lemma}\label{lem:maximal_faces}
Let $M$ be a JBW-algebra. With respect to inclusion, the maximal norm closed proper faces of $2B_{e_\perp} \subseteq M_*$ are of the form $F_p-F_{p^\perp}$ for some nontrivial projection $p$.
\end{lemma}

\begin{proof}
Suppose that $F=F_p-F_{p^\perp}$ for some nontrivial projection $p$. If $F_p-F_{p^\perp}\subseteq F_q-F_r$ for some orthogonal projections $q$ and $r$, then we must have $p\le q$ and $p^\perp\le r$. But $e=p+p^\perp\le q+r\le e$, so $q=r^\perp$ and we have $r^\perp\le p\le q=r^\perp$. Hence $p=q$ and $p^\perp=r$, so $F$ is maximal.

Conversely, if $F=F_p-F_q$ with $p^\perp\neq q$, then $q\lneq p^\perp$ so $F_q\subsetneq F_{p^\perp}$ by \cite[Corollary~5.33]{AS2}. Hence $F_p-F_q\subsetneq F_p-F_{p^\perp}$, showing that $F$ is not maximal.
\end{proof}

We will denote these maximal faces by $G_p:=F_p-F_{p^\perp}$. By \cite[Proposition~5.39]{AS2} the faces $F_u$ of $K$ for an atom $u$ are of the form $F_u=\{\rho\}$ for some pure state $\rho\in K$. We will follow the notation in \cite[Definition~5.40]{AS2} and denote this pure state $\rho$ by $\hat{u}$.  

\begin{lemma}\label{lem:mutual_distance_2}
Let $M$ be a JBW-algebra, and let $F=F_p-F_q\subseteq 2B_{e_\perp}$ be a proper face. Then $\norm{\varphi-\psi}\le 2$ for all $\varphi,\psi\in F$ if and only if $p$ or $q$ is an atom.
\end{lemma}

\begin{proof}
Let $p$ be an atom. Then $F_p=\{\hat{p}\}$, so if $\varphi,\psi\in F$, then $\varphi=\hat{p}-\tau$ and $\psi=\hat{p}-\sigma$ for some $\tau,\sigma\in F_q$. It follows that
\[
\norm{\varphi-\psi}=\norm{(\hat{p}-\tau)-(\hat{p}-\sigma)}=\norm{\sigma-\tau} \le 2.
\]

Conversely, suppose that $p=p_1+p_2$ and $q=q_1+q_2$ for some nontrivial projections $p_1$, $p_2$, $q_1$, and $q_2$. Now choose $\varphi_i\in F_{p_i}$ and $\psi_i\in F_{q_i}$ for $i=1,2$. Note that $\varphi_1\perp\varphi_2$ and $\psi_1\perp\psi_2$, so $(\varphi_1+\psi_2)\perp(\varphi_2+\psi_1)$ by    Lemma~\ref{lem:orthogonal_functionals1}. For $\varphi:=\varphi_1-\psi_1\in F_p-F_q$ and $\psi:=\varphi_2-\psi_2\in F_p-F_q$ it now follows that
\[
\norm{\varphi-\psi}=\norm{(\varphi_1-\psi_1)-(\varphi_2-\psi_2)}
=\norm{(\varphi_1+\psi_2)-(\varphi_2+\psi_1)}=\norm{\varphi_1+\psi_2}+\norm{\varphi_2+\psi_1}=4.
\]
\end{proof}
Lemma~\ref{lem:mutual_distance_2} shows that $\pm G_u$ are precisely the maximal proper norm closed faces of $2B_{e_\perp}$ in which elements have mutual distance at most two.

\subsection{The orthoisomorphism on the atomic parts}
The maximal norm closed faces of diameter at most two in $2B_{e_\perp}$, which equal $\pm G_u$ by the above, must be preserved by the isometric isomorphism $S' \colon e_B^\perp \to e_A^\perp$. We proceed with showing that such an isometric isomorphism fixes the sign of these faces, and consequently induces an orthoisomorphism on the projection lattices between $(A^{**})_a$ and $(B^{**})_a$.

\begin{lemma}\label{lem:sign_theta}
Let $M$ and $N$ be atomic JBW-algebras and let
\[
L \colon e^M_\perp \subseteq M_* \to e^N_\perp \subseteq N_*
\]
be an isometric isomorphism. Then there exists an $\eps \in \{\pm 1\}$ and a map $\theta$ from the atoms of $M$ into the atoms of $N$ such that $L(G_u)=\eps G_{\theta(u)}$ for all atoms $u \in M$.
\end{lemma}

\begin{proof} 
By the above discussion $L(G_u)$ must be of the form $\eps(u) G_{\theta(u)}$ for some $\eps(u) \in \{\pm 1\}$ and some atom $\theta(u) \in N$. We have to show that $\eps$ can be chosen uniformly. Note that if $N$ is either $\R^2$ or a spin factor, then $-G_{\theta(u)} = G_{\theta(u)^\perp}$ for all atoms $u$, so $\eps$ can be chosen uniformly. So we may assume that $N$ is neither $\R^2$ nor a spin factor. By Lemma~\ref{lem:equivalences_R2_and_SF}, if $p \in \P(N)$, then $p$ and $p^\perp$ cannot both be atoms, so the norm closed faces $G_p \subseteq 2B_{e_\perp^N}$ are not extreme points by Corollary~\ref{cor:extreme_points_dual_JB}. Hence there are norm closed faces in $2B_{e_\perp^M}$ that are not extreme points; therefore $M$ is neither $\R^2$ nor a spin factor as well.

Suppose that $u,v \in M$ are distinct atoms such that $L(G_u)=G_r$ and $L(G_v)=-G_s$ for some atoms $r,s \in N$. Then 
\[
F_{r \wedge s} - F_{r^\perp \wedge s^\perp} = G_r \cap G_s = L(G_u) \cap L(-G_v) = L(G_u \cap -G_v)=L(F_{u \wedge v^\perp} - F_{u^\perp\wedge v}),
\]
and since $r$ and $s$ are atoms, either $r \wedge s = 0$ or $r\wedge s=r=s$. In the second case $F_{r \wedge s} - F_{r^\perp \wedge s^\perp} = G_r$ and so $F_{u \wedge v^\perp} - F_{u^\perp\wedge v} = G_u$, which implies that $v=u^\perp$. Lemma~\ref{lem:equivalences_R2_and_SF} now implies that $M$ is either $\R^2$ or a spin factor, which contradicts our assumption. In the first case $F_{r \wedge s}=\emptyset$ and so $F_{u \wedge v^\perp} - F_{u^\perp\wedge v} = \emptyset$ as well.
 
If $u \perp v$, then
\[
F_{u \wedge v^\perp} - F_{u^\perp\wedge v} = F_u - F_v = \{\hat{u} - \hat{v} \} \not= \emptyset,
\]
which is impossible, showing that $\eps(u) = \eps(v)$ for orthogonal atoms $u$ and $v$. Suppose that $u$ and $v$ are not orthogonal. If $u \vee v = e$, then Lemma~\ref{lem:equivalences_R2_and_SF} shows that $M$ is either $\R^2$ or a spin factor, contradicting our assumption. Hence $u \vee v < e$. Choose an atom $w \leq (u \vee v)^\perp$. Then $w \perp u$ and $w \perp v$, so $\eps(u) = \eps(w) = \eps(v)$.
\end{proof}

Fixing the sign for such an isometry $L$ allows us to construct an orthoisomorphism between the projection lattices of atomic JBW-algebras.

\begin{theorem}\label{thm:orthoisomorphism_atomic_part}
Let $M$ and $N$ be atomic JBW-algebras and let
\[
L\colon e^M_\perp\subseteq M_*\to e^N_\perp\subseteq N_*
\]
be an isometric isomorphism. Let $\eps \in \{\pm 1\}$ and $\theta$ be a map from the atoms of $M$ into the atoms of $N$ such that $L(G_u)=\eps G_{\theta(u)}$ for all atoms $u \in M$. Then we can extend $\theta$ to
\[
\theta\colon\P(M) \to \P(N), \quad \theta(p) := \bigvee \{ \theta(u) \colon u \mbox{ an atom with } u \le p \},
\]
which defines an orthoisomorphism. Moreover, if $p, q \in \P(M)$ are orthogonal nontrivial projections, then $L(G_p) = \eps G_{\theta(p)}$ and $L( F_p - F_q) = \eps (F_{\theta(p)} - F_{\theta(q)})$.
\end{theorem}

\begin{proof}

Since $\bigvee\emptyset=0$, we have $\theta(0)=0$. Let $u$ and $v$ be orthogonal atoms in $M$. We have $G_u \cap - G_v = F_u - F_v = \{ \hat{u} - \hat{v} \}$, so 
\begin{align*}
\eps(F_{\theta(u) \wedge \theta(v)^\perp} - F_{\theta(u)^\perp \wedge \theta(v)}) 
&= \eps(G_{\theta(u)}\cap-G_{\theta(v)}) = L(G_{u})\cap-L(G_{v}) \\
&= L(G_{u}\cap-G_{v}) = L(\{\hat{u} - \hat{v} \}).
\end{align*}
Hence $\theta(u)\wedge \theta(v)^\perp=\theta(u)$ and $\theta(u)^\perp\wedge \theta(v)=\theta(v)$, thus $\theta(u)$ and $\theta(v)$ are orthogonal atoms in $N$. Similarly, the converse is true, so $u,v\in M$ are orthogonal atoms if and only if $\theta(u),\theta(v)\in N$ are orthogonal atoms. Since orthogonality is preserved under suprema, the extended $\theta$ also preserves orthogonality.

Let $p \in \P(M)$ be nonzero. We claim that $\theta$ induces a bijection from the atoms $u \leq p$ to the atoms $v \leq \theta(p)$. Indeed, if $u \leq p$ is an atom, then $\theta(u) \leq \theta(p)$ by definition of $\theta(p)$. Conversely, let $\theta(u) \leq \theta(p)$ be an atom for some atom $u \in \P(M)$, then we have to show that $u \leq p$. Since $\theta(p) \perp \theta(p^\perp)$, it follows that $\theta(u) \perp \theta(p^\perp)$, and so $\theta(u) \perp \theta(w)$ for all atoms $w \leq p^\perp$. Hence $u \perp w$ for all atoms $w \leq p^\perp$, which implies that $u \perp p^\perp$, as required.

Let $\xi\colon \P(N) \to \P(M)$ be defined by $\xi(q):=\bigvee\{\theta^{-1}(v)\colon v\mbox{ an atom with }v\le q\}$. It follows from the claim that
\[
\xi\circ\theta(p)=\bigvee\{\theta^{-1}(v)\colon v\mbox{ an atom with }v\le \theta(p)\}=\bigvee\{u\colon u\mbox{ an atom with }u\le p\}=p.
\] 
Similarly, we find that $\theta\circ\xi(q)=q$ for all nonatomic and nontrivial $q\in\P(N)$, so $\xi=\theta^{-1}$. We conclude that $\theta$ is a bijection such that $\theta$ and $\theta^{-1}$ preserve orthogonality, so $\theta$ is an orthoisomorphism.

Let $p\in\P(M)$. Then $L(G_p)$ is a maximal norm closed proper face and so there exists a $q\in\P(N)$ such that $L(G_p) = \eps G_q$. If $u\le p$ and $v\le p^\perp$ are atoms, then $G_{u}\cap-G_{v}=\{\hat{u} - \hat{v} \}$ is an extreme point of $G_p$, so $L$ must map it to an extreme point of $\eps G_q$. By definition of $\theta$ on the atoms in $\P(M)$, this extreme point is $\eps(G_{\theta(u)}\cap-G_{\theta(v)})$, so $\theta(u)\le q$ and $\theta(v)\le q^\perp$ are atoms. Similarly, every atom $w\le q$ and $s \le q^\perp$ is of the form $w=\theta(u)$ and $s = \theta(v)$ for atoms $u\le p$ and $v\le p^\perp$. Hence
\[
q=\bigvee\{w\colon w\le q\mbox{ an atom}\}=\bigvee\{\theta(u)\colon u\le p\mbox{ an atom}\}=\theta(p).
\]
This together with the fact that $\theta$ is an orthoisomorphism implies that
\[
L(F_{p}-F_{q})=L(G_p\cap-G_q)=L(G_{p})\cap-L(G_{q})= \eps(G_{\theta(p)}\cap-G_{\theta(q)}) = \eps(F_{\theta(p)}-F_{\theta(q)})
\]
for any nontrivial orthogonal projections $p,q\in\P(M)$.
\end{proof}

\section{Jordan structure of variation norm isometries}\label{sec:4}

In this section we will show that the orthoisomorphism induced by an isometric isomorphism $L\colon e_\perp^M\subseteq M_*\to e_\perp^N\subseteq N_*$ form Theorem~\ref{thm:orthoisomorphism_atomic_part} extends to a Jordan isomorphism $J\colon M \to N$. Furthermore, we show that this Jordan isomorphism will characterize the variation norm isometries form $[A]$ to $[B]$ for unital JB-algebras $A$ and $B$. 

\subsection{The Jordan isomorphism induced by the orthoisomorphism $\theta$}

We start by investigating the relation between the orthoisomorphism $\theta$ and the adjoint of the isometric isomorphism $L$.

\begin{lemma}\label{lem:norming_projection_G_p}
Let $M$ be a JBW-algebra and $p\in M$ be a nontrivial projection. Then $[p]\in [M]$ is the unique class of a projection that attains the norm of $G_p \subseteq 2B_{e_\perp}$, i.e., $[p](G_p)=1$.
\end{lemma}

\begin{proof}
It is clear that $[p](G_p)=1$. Suppose that $q\in M$ is a nontrivial projection such that $[q](G_p)=1$. Then for all $\phi \in F_p$ and $\psi \in F_{p^\perp}$, 
\[
1=(\varphi-\psi)({\textstyle\frac{1}{2}}(q-q^\perp))=\varphi({\textstyle\frac{1}{2}}(q-q^\perp))-\psi({\textstyle\frac{1}{2}}(q-q^\perp))\le{\textstyle\frac{1}{2}}\varphi(q)+{\textstyle\frac{1}{2}}\psi(q^\perp)\le 1,
\]
 so $\varphi(q)=\psi(q^\perp)=1$, hence $s(\phi) \le q$ and $s(\psi) \le q^\perp$. By taking suprema over all $\phi \in F_p$ and $\psi \in F_{p^\perp}$ we obtain $p \leq q$ and $p^\perp \leq q^\perp$, from which we conclude that $p=q$.
\end{proof}

\begin{lemma}\label{lem:link_orthomorphism_L*}
Let $M$ and $N$ be JBW-algebras. If 
\[
L\colon e^M_\perp\subseteq M_*\to  e^N_\perp\subseteq N_* 
\]
is an isometric isomorphism, $\eps \in \{\pm 1\}$, and $\theta\colon\P(M)\to\P(N)$ is an orthoisomorphism such that $L(G_p)=\eps G_{\theta(p)}$ for all $p \in \P(M)$, then $[\theta(p)] = \eps (L^{-1})^*[p]$ for all $p \in \P(M)$.
\end{lemma}

\begin{proof}
The statement clearly holds for the trivial projections, so let $p \in \mathcal{P}(M)$ be nontrivial. The extreme points of the unit balls in $[M]$ and $[N]$ are the equivalence classes of nontrivial projections by \cite[Lemma~4.1]{LRW2} and these must be bijectively mapped onto one another by the isometric isomorphism $\eps (L^{-1})^*$. Hence $\eps (L^{-1})^*[p]$ is the class of a nontrivial projection in $M$. Moreover
\[
\eps (L^{-1})^*[p](G_{\theta(p)}) = [p](L^{-1} (\eps G_{\theta(p)})) = [p](G_p) = 1,
\]
so it follows from Lemma~\ref{lem:norming_projection_G_p} that $\eps (L^{-1})^*[p] = [\theta(p)]$.
\end{proof}

As already shown in \cite[Section~4.3]{LRW2}, the above orthoisomorphism can now be extended to a Jordan isomorphism. For the convenience of the reader we will include the proof, which is in the appendix since dealing with the type $I_2$ part is long and technical.

\begin{theorem}\label{thm:Jordan_isomorphism_from_theta}
Let $M$ and $N$ be JBW-algebras. If 
\[
L\colon e^M_\perp\subseteq M_*\to e^N_\perp\subseteq N_* 
\]
is an isometric isomorphism, $\eps \in \{\pm 1\}$, and $\theta\colon\P(M)\to\P(N)$ is an orthoisomorphism such that $L(G_p)=\eps G_{\theta(p)}$ for all $p \in \P(M)$, then $\theta$ extends to a Jordan isomorphism $J\colon M \to N$. 
\end{theorem}

\subsection{Characterization of variation norm isometries}

We are now ready to characterize the surjective variation norm isometries between JB-algebras. 

\begin{theorem}\label{thm:char_var_isoms}
If $A$ and $B$ are unital JB-algebras, then $S \colon [A] \to [B]$ is a surjective linear variation norm isometry if and only if 
\begin{equation}\label{eq:char_var_isoms}
S[x] =  \eps [J][x] \quad \text{for all } x \in A
\end{equation}
for some $\eps \in \{\pm 1\}$ and a Jordan isomorphism $J \colon A \to B$.
\end{theorem}

\begin{proof}
The adjoint of $S$ yields an isometric isomorphism $S' \colon e^\perp_B \to e^\perp_A$ and after multiplying $S'$ by $\eps\in\{\pm 1\}$, we obtain an orthoisomorphism $\theta\colon\P((A^{**})_a) \to \P((B^{**})_a)$ by Theorem~\ref{thm:orthoisomorphism_atomic_part}. 
Lemma~\ref{lem:link_orthomorphism_L*} and Theorem~\ref{thm:Jordan_isomorphism_from_theta} in turn imply that $\theta$ extends to a Jordan isomorphism $J \colon (A^{**})_a \to (B^{**})_a$ and $[\theta(p)]=\eps S'^*[p]$ for all $p\in (A^{**})_a$. 
The corresponding quotient map $[J]\colon [(A^{**})_a]\to [(B^{**})_a]$ satisfies $[J][p]:=[Jp]=\eps S'^*[p]$ for all $p\in (A^{**})_a$ and so $[J]=\eps S'^*$ on the equivalence classes of the projections in $[(A^{**})_a]$. As the variation norm is dominated by twice the JB-norm, the linear span of the equivalence classes of the projections is variation norm dense in $[(A^{**})_a]$ by the spectral theorem. Hence $[J]=\eps S'^*|_{[(A^{**})_a]}$. 

If $\phi$ is a pure state on $A$, then its support projection is an atom and hence contained in $(A^{**})_a$. From this it follows easily that if $x \in A$, then the value of $\phi(x)$ is independent of whether we consider $x$ as an element of $A$ or $(A^{**})_a$. Therefore, for $\varphi-\psi\in\mathrm{ext}(B_{e^\perp})$ we find that
\[
(\varphi-\psi)(\eps S[x]) = \eps S'(\varphi-\psi)([x]) = \eps S'^*[x](\varphi-\psi) = (\varphi-\psi)([J][x])
\] 
for all $[x]\in [A]$. So the Krein-Milman theorem implies that $\eps S=[J]$ on $[A]$. Since we can view $A$ and $B$ as subalgebras of $(A^{**})_a$ and $(B^{**})_a$, respectively, the restriction $J|_A\colon A\to B$ is a Jordan isomorphism.

Conversely, any map of the form \eqref{eq:char_var_isoms} is obviously a variation norm isometry.
\end{proof}

The following argument now shows that Theorem~\ref{thm:char_var_isoms} implies Hamhalter's characterization.

\begin{lemma}\label{lem:ham_char}
If $A$ and $B$ are unital JB-algebras, then $T \colon A \to B$ is a surjective linear variation norm isometry if and only if 
\[ 
Tx =  \eps Jx + \phi(x)e \quad \text{for all } x \in A
\]
for some $\eps \in \{\pm 1\}$, a Jordan isomorphism $J \colon A \to B$ and a linear functional $\phi \colon A \to \R$.
\end{lemma}
\begin{proof}
Let $T \colon A \to B$ be a surjective linear variation norm isometry. Then $T$ leaves $\Sp(e)$ invariant, and so $[T] = \eps [J]$ for some $\eps \in \{\pm 1\}$ and some Jordan isomorphism $J \colon A \to B$ by Theorem~\ref{thm:char_var_isoms}. Hence $[T-\eps J] =0$, showing that $T- \eps J$ maps $A$ into $\Sp(e)$. Therefore
\[ T = \eps J + (T - \eps J) = \eps J + \phi \otimes e \]
for some  linear functional $\phi \colon A \to \R$, as required.

Conversely, any $T$ of the above form is obviously a surjective linear variation norm isometry.
\end{proof}

If $A$ and $B$ are JB-algebras without a unit, then the variation norm is a genuine norm on $A$ and $B$. If $S\colon A\to B$ is a surjective variation norm isometry and $A_e$ and $B_e$ are the unitizations of $A$ and $B$ respectively, then $S$ extends to  $\tilde{S}\colon A_e \to B_e$ by putting $\tilde{S}(a+\lambda e):=Sa+\lambda e$. By construction, $\tilde{S}$ induces the quotient map $[\tilde{S}]\colon [A_e]\to [B_e]$ which is a variation norm isometry.  Since $\pi_A\colon [A_e] \to A$ and $\pi_B\colon [B_e] \to B$ defined by $[x]\mapsto x$ are surjective linear variation norm isometries, it follows that $S=\pi_B\circ [\tilde{S}]\circ \pi_A^{-1}$ and this yields the following characterization for variation norm isometries on nonunital JB-algebras.

\begin{corollary}\label{cor:char_var_isoms_nonunital}
If $A$ and $B$ are nonunital JB-algebras, then $S \colon A \to B$ is a surjective linear variation norm isometry if and only if 
\begin{equation*}
Sx =  \eps Jx \quad \text{for all } x \in A
\end{equation*}
for some $\eps \in \{\pm 1\}$ and a Jordan isomorphism $J \colon A \to B$.
\end{corollary} 

\section{Hilbert's metric isometries on JB-algebras}\label{sec:5}

We can now prove the following characterization of the surjective Hilbert's metric isometries between cones in unital JB-algebras. 

\begin{theorem}\label{thm:char_Hilbert_isoms}
If $A$ and $B$ are unital JB-algebras, then $f\colon\ol{A}^\circ_+\to\ol{B}^\circ_+$ is a surjective Hilbert's metric isometry if and only if 

\begin{equation}\label{eq:Hilbert_isoms}
f(\ol{x})=\ol{U_y J(x^\eps)}\quad\mbox{for all }\ol{x}\in\ol{A}^\circ_+,
\end{equation}
where $\eps\in\{\pm 1\}$, $y\in A^\circ_+$, and $J\colon A\to B$ is a Jordan isomorphism. In this case $y\in f(\ol{e}_A)^{\frac{1}{2}}$.
\end{theorem}

\begin{proof}
Let $ f \colon \ol{A}_+^\circ \to \ol{B}_+^\circ$ be a surjective Hilbert's metric isometry. Then we can define a new surjective isometry $g \colon \ol{A}_+^\circ \to \ol{B}_+^\circ$ by
\[
g(\ol{x}) := U_{f(\ol{e}_A)^{-\frac{1}{2}}} f(\ol{x})\mbox{\quad for all }\ol{x}\in \ol{A}_+^\circ.
\]
Note that $g(\ol{e}_A)=\ol{e}_B$ and by \cite[Theorem~2.17]{LRW2} the map $S:=\log\circ g \circ \exp \colon [A] \to [B]$ is an isometric isomorphism for the variation norm. Hence $S = \eps [J]$ by Theorem~\ref{thm:char_var_isoms} for some $\eps\in\{\pm 1\}$ and some Jordan isomorphism $J\colon A\to B$. Note that $J$ induces a map from $\ol{A}^\circ_+$ to $\ol{B}_+^\circ$. Let $x \in A_+^\circ$, then $x = \exp(z)$ for some $z \in A$, and so 
\begin{eqnarray*}
g(\ol{x})^\eps  =  \exp([J]\log(\ol{\exp (z)})) = \exp([J][z]) = \exp([Jz]) = \ol{\exp(Jz)} = \ol{J(\exp(z))} = \ol{J x} = J\ol{x}.
\end{eqnarray*}
Thus, 
\[
(U_{f(\ol{e}_A)^{-\frac{1}{2}}} f(\ol{x}))^{\eps} = J \ol{x}\mbox{\quad for all $\ol{x}\in \ol{A}_+^\circ$},
\]
hence
\[
f(\ol{x}) = U_{f(\ol{e}_A)^{\frac{1}{2}}} (J\ol{x})^{\eps} =U_{f(\ol{e}_A)^{\frac{1}{2}}} (\ol{Jx})^{\eps}= U_{f(\ol{e}_A)^{\frac{1}{2}}} \ol{J(x^{\eps})}=\ol{U_yJ(x^\eps)}
\]
for some $y\in f(\ol{e}_A)^{\frac{1}{2}}$. To complete the proof note that any map  of the form (\ref{eq:Hilbert_isoms}) is a surjective Hilbert's metric isometry.
\end{proof}
Theorem~\ref{thm:char_Hilbert_isoms} has the following direct consequence.

\begin{corollary}
For unital JB-algebras $A$ and $B$ the metric spaces $(\ol{A}^\circ_+,d_H)$ and $(\ol{B}^\circ_+,d_H)$ are isometric if and only if $A$ and $B$ are Jordan isomorphic.
\end{corollary}

Next, we will describe the isometry group $\mathrm{Isom}(\ol{A}_+^\circ)$ consisting of all surjective Hilbert's metric isometries on $\ol{A}_+^\circ$. A map $\tau\colon\ol{A}_+^\circ\to\ol{A}_+^\circ$ that is of the form $\tau(\ol{x})=\ol{Tx}$ where $T\in\mathrm{Aut}(A_+)$ is called a \textit{projectivity} of $A_+$ and the collection of projectivities $\mathrm{Proj}(A_+)$ form a subgroup of $\mathrm{Isom}(\ol{A}_+^\circ)$. By \cite[Proposition~2.3]{LRW2} these projectivities are of the form $\tau(\ol{x})=\ol{U_yJx}$ for some $y\in A_+^\circ$ and a Jordan isomorphism $J$. Moreover, the subgroup $C_2$ of order two which is generated by the inversion map $\iota$ acts on $\mathrm{Proj}(A_+)$ via conjugation as
\begin{equation*}
(\iota\circ\tau\circ\iota)(\ol{x})=\ol{U_yJx^{-1}}^{-1}=\ol{U_{y^{-1}}(Jx^{-1})^{-1}}=\ol{U_{y^{-1}}Jx},
\end{equation*}
so $\iota\circ\tau\circ\iota\in\mathrm{Proj}(A_+)$ and hence $\mathrm{Proj}(A_+)$ is a normal subgroup of $\mathrm{Isom}(\ol{A}_+^\circ)$. By Theorem~\ref{thm:char_Hilbert_isoms} we can write every element of $\mathrm{Isom}(\ol{A}_+^\circ)$ as the product of an element in $\mathrm{Proj}(A_+)$ with an element in $C_2$, and the characterization of the isometry group now follows immediately from Lemma~\ref{lem:equivalences_R2_and_SF}.

\begin{theorem}\label{thm:isom_group}
Let $A$ be a unital JB-algebra. Then the group of surjective Hilbert's metric isometries $\mathrm{Isom}(\ol{A}^\circ_+,d_H)$ satisfies $\mathrm{Isom}(\ol{A}^\circ_+,d_H)\cong\mathrm{Proj}(A_+)$ if and only if $A$ is either $\R^2$ or a spin factor. In all other cases $\mathrm{Isom}(\ol{A}^\circ_+,d_H)\cong\mathrm{Proj}(A_+)\rtimes C_2$ where $C_2$ is the group of order two generated by the inversion map $\iota$.
\end{theorem}

\begin{remark}
In view of Theorem~\ref{thm:char_var_isoms} and Theorem~\ref{thm:char_Hilbert_isoms} we would like to remark that since the $\log$ and $\exp$ functions commute with Jordan isomorphisms, it follows that $f\colon \ol{A}^\circ\to \ol{B}^\circ$ is a surjective Hilbert's metric isometry  if and only if $\exp\circ f\circ\log\colon[A]\to[B]$ is a surjective variation norm isometry. In other words, there is a bijective correspondence between surjective Hilbert's metric isometries and surjective variation norm isometries.
\end{remark}

\subsection*{Appendix: proof of Theorem~\ref{thm:Jordan_isomorphism_from_theta}}

\begin{theorem*}
Let $M$ and $N$ be JBW-algebras. If 
\[
L\colon e^M_\perp\subseteq M_*\to e^N_\perp\subseteq N_* 
\]
is an isometric isomorphism, $\eps \in \{\pm 1\}$, and $\theta\colon\P(M)\to\P(N)$ is an orthoisomorphism such that $L(G_p)=\eps G_{\theta(p)}$ for all $p \in \P(M)$, then $\theta$ extends to a Jordan isomorphism $J\colon M \to N$. 
\end{theorem*}

\begin{proof}
By the proof of \cite[Lemma~1]{Dye}, $\theta$ is an order isomorphism and preserves products of operator commuting projections. Write $M=M_2\oplus \tilde{M}$ and $N=N_2\oplus \tilde{N}$ where $M_2$ and $N_2$ are type $I_2$ direct summands, and $\tilde{M}$ and $\tilde{N}$ are JBW-algebras without type $I_2$ direct summands. See \cite[Theorem~5.1.5, Theorem~5.3.5]{HO}. If $\tilde{p}\in\P(M)$ and $\tilde{q}\in\P(N)$ are the central projections such that $\tilde{p}M=\tilde{M}$ and $\tilde{q}N=\tilde{N}$, then $\theta(\tilde{p})=\tilde{q}$ and the restriction $\theta|_{\P(\tilde{p}M)}$ is an orthoisomorphism from $\P(\tilde{M})$ to $\P(\tilde{N})$. By \cite[Corollary~2]{BW} (which also holds for JBW-algebras) this orthoisomorphism extends to a Jordan isomorphism $\tilde{J}\colon \tilde{M}\to\tilde{N}$.

Next, we show that the orthoisomorphism $\theta|_{\P(M_2)}\colon \P(M_2)\to\P(N_2)$ extends to a Jordan isomorphism as well. By \cite[Theorem~2]{S} we can represent
\[
M_2\cong \bigoplus_k L^\infty(\Omega_k,V_k)\quad\mbox{and}\quad
N_2\cong \bigoplus_l L^\infty(\Xi_l, V_l)
\]
where $k,l$ are cardinals, $\Omega_k$ and $\Xi_l$ are measure spaces, and $V_i=H_i\oplus\R$ are spin factors with $\dim H_i=i$. We denote the unit in each $V_k$ by $u$. Let $\Omega := \bigsqcup_k \Omega_k$. By identifying $f\in L^\infty(\Omega)$ with $\omega\mapsto f(\omega) u$, we can view $L^\infty(\Omega)$ as lying inside $M_2$. The center $Z(M_2)$ of $M_2$ equals $L^\infty(\Omega)$ and if $p:= \mathbf{1}_A \in Z(M_2)$, then $Z(pM_2) = L^\infty(A)$. As $\theta$ preserves products of operator commuting projections, it preserves the center, and it is straightforward to see that the restriction $\theta|_{\P(Z(M_2))} \colon \P(Z(M_2)) \to \P(Z(N_2))$ extends to a Jordan isomorphism $T \colon Z(M_2) \to Z(N_2)$.

Let $x\in M_2$. For almost all $\omega\in\Omega$ the element $x(\omega)$ has rank 1 or rank 2, so modulo null sets we can write $\Omega $ as $\Omega=\Omega^1 \sqcup \Omega^2$ where
\[
\Omega^i:=\left\{\omega\in\Omega\colon \#\sigma(x(\omega))=i\right\}.
\]
If we write $q_i:=\mathbf{1}_{\Omega^i}$ for $i=1,2$, then there exist unique $\alpha\in Z(q_1 M_2)$, $\beta,\gamma\in Z(q_2 M_2)$, and $p \in \P(q_2M_2)$ with $p(\omega)$ of rank 1 a.e.\ such that
\[
x(\omega):=
\begin{cases}
\alpha(\omega)u & \mbox{ if $\omega \in \Omega^1$}\\
\beta(\omega)p(\omega)+\gamma(\omega)p(\omega)^\perp& \mbox{ if $\omega \in \Omega^2$}\end{cases}
\]
which yields $x=\alpha+\beta p+\gamma p^\perp$ as a unique representation. Define $J_2\colon M_2\to N_2$ by
\[
J_2(x):=T\alpha +T\beta \theta(p)+T\gamma \theta(p)^\perp.
\]

Since $\theta$ preserves central projections and orthogonality, it maps a.e.\ rank 1 projections to a.e.\ rank 1 projections by \cite[Lemma~4.17]{LRW2}. Now $x \in \P(M_2)$ if and only if $\alpha,\beta,\gamma\in\P(Z(M_2))$, and in this case, since $T$ extends $\theta|_{\P(Z(M_2))}$,
\[
J_2(x)=T\alpha+T\beta \theta(p)+T\gamma \theta(p)^\perp=\theta(\alpha)+\theta(\beta)\theta(p)+\theta(\gamma)\theta(p)^\perp=\theta(\alpha)+\theta(\beta p)+\theta(\gamma p^\perp)=\theta(x)
\]
as $\theta$ preserves products of operator commuting projections and is an order isomorphism. So $J_2$ extends $\theta|_{\P(M_2)}$. Now the hard part is to show that $J_2$ is linear.

For $\mu\in\R$ and the unit $e_2\in M_2$ we have $J_{2}(x+\mu e_2)=J_2(x)+\mu e_2$, so $J_2$ induces the quotient map $[J_2]\colon[M_2]\to[N_2]$ defined by $[J_2]([x]):=[J_2 x]$. We claim that $[J_2]$ coincides with the linear map $\eps(L^{-1})^*$ on $[M_2]$. Indeed, let $x \in M_2$ be such that $x=\alpha+\beta p+\gamma p^\perp$ where $\alpha=\sum_i\alpha_i\mathbf{1}_{A_i},\beta=\sum_j\beta_j\mathbf{1}_{B_j}$, and
$\gamma=\sum_k\gamma_k\mathbf{1}_{C_k}$ are step functions. Since $\theta$ preserves products of operator commuting projections and the fact that $T$ maps step functions to step functions, Lemma~\ref{lem:link_orthomorphism_L*} implies that
\begin{align*}
[J_2]([x])&=[J_2(x)]=[T\alpha+T\beta \theta(p)+T\gamma \theta(p)^\perp] \\
&=\sum_i\alpha_i[\theta(\mathbf{1}_{A_i})]+\sum_j\beta_j[\theta(\mathbf{1}_{B_j} p)]+\sum_k\gamma_k[\theta(\mathbf{1}_{C_k} p^\perp)] \\ 
&=\sum_i\alpha_i  \eps(L^{-1})^*[\mathbf{1}_{A_i}]+\sum_j\beta_j  \eps(L^{-1})^*[\mathbf{1}_{B_j} p]+\sum_k\gamma_k \eps(L^{-1})^*[\mathbf{1}_{C_k} p^\perp] \\
&= \eps(L^{-1})^*[x].
\end{align*}
Now, for general $x=\alpha+\beta p+\gamma p^\perp\in M_2$ let $\alpha'$, $\beta'$, and $\gamma'$ be approximating step functions for $\alpha$, $\beta$, and $\gamma$. If we put $y:=\alpha'+\beta' p+\gamma' p^\perp$, then
\begin{align*}
\|x-y\|&\le\|\alpha-\alpha'\|+\|\beta-\beta'\|+
\|\gamma-\gamma'\|
\end{align*}
and
\begin{align*}
\|J_2(x)-J_2(y)\|&\le\|\alpha-\alpha'\|+\|\beta-\beta'\|+
\|\gamma-\gamma'\|
\end{align*}
as $T$ is an isometry, so both norms can be made arbitrarily small. This implies that
\begin{align*}
\|[J_2]([x]) - \eps(L^{-1})^*[x] \|_q &\le \|[J_2]([x])-[J_2]([y])\|_q
+ \|[J_2]([y]) - \eps(L^{-1})^*[y]\|_q \\
&\quad + \|\eps(L^{-1})^*[y] - \eps(L^{-1})^*[x]\|_q \\
&= \|[J_2]([x])-[J_2]([y])\|_q + \| \eps(L^{-1})^*([y]-[x])\|_q \\
&\le \|[J_2(x)-J_2(y)]\|_q + \|[y-x]\|_q \\
&\le 2\|J_2(x)-J_2(y)\| + 2\|y-x\|
\end{align*}
can be made arbitrarily small, and we conclude that $[J_2] = \eps(L^{-1})^*$ on $[M_2]$.

Let $\phi$ be a state on $Z(N_2)=L^\infty(\Xi)$. Then $T^*\phi$ is a state on $Z(M_2)= L^\infty(\Omega)$, and define the functionals $\mathrm{tr}\otimes T^*\phi\in M_2^*$ and $\mathrm{tr}\otimes\phi\in N_2^*$ by

\[
(\mathrm{tr}\otimes T^*\phi)(x):=T^*\phi(\omega\mapsto\mathrm{tr}(x(\omega)))\quad\mbox{and}\quad
(\mathrm{tr}\otimes \phi)(y):=\phi(\xi\mapsto\mathrm{tr}(y(\xi))).
\]
Put $M_0:=\ker (\mathrm{tr}\otimes T^*\phi)$ and $N_0:= \ker (\mathrm{tr}\otimes \phi)$. Since $e_2\notin M_0$ and $e_2\notin N_0$, the corresponding quotient maps $\pi_M \colon M_0 \to [M_2]$ and $\pi_N \colon N_0 \to [N_2]$ are linear isomorphisms. Furthermore, we have that $J_2(M_0)\subseteq N_0$. Indeed, if $x\in M_2$, then since $\theta(p)$ is a.e.\ rank 1, 
\[
(\mathrm{tr}\otimes \phi)(J_2(x))=(\mathrm{tr}\otimes \phi)(T\alpha+T\beta \theta(p)+T\gamma \theta(p)^\perp)=\phi(2T\alpha+T\beta+T\gamma).
\]
Therefore, for $x\in M_0$ it follows that 
\begin{align*}
(\mathrm{tr}\otimes \phi)(J_2(x))&=\phi(2T\alpha+T\beta+T\gamma)=
\phi(T(2\alpha+\beta+\gamma))\\&=T^*\phi(2\alpha+\beta+\gamma)=(\mathrm{tr}\otimes T^*\phi)(x)\\&=0.
\end{align*}
Now, if $x\in M_0$, then $J_2(x)\in N_0$ which shows the last equality of the equation
\[
\pi_N^{-1}\circ[J_2]\circ \pi_M(x) = \pi_N^{-1}[J_2][x] = \pi_N^{-1}[J_2(x)]=J_2(x),
\]
hence $J_2|_{M_0}$ is linear. As $M_2=M_0\oplus \Sp(e_2)$ and $N_2=N_0\oplus \Sp(e_2)$, and we have $J_2(x+\mu e_2)=J_2(x)+\mu e_2$ for all $\mu\in\R$, it follows that $J_2=J_2|_{M_0}\oplus\mathrm{Id}_{\Sp(e_2)}$ is linear.

Moreover, we have
\begin{align*}
\|x\|&=\esssup_{\omega\in\Omega}\|x(\omega)\|=\max\{\|\alpha\|_\infty,
\|\beta\|_\infty,\|\gamma\|_\infty\}\\&
=\max\{\|T\alpha\|_\infty,
\|T\beta\|_\infty,\|T\gamma\|_\infty\}\\&=
\esssup_{\xi\in\Xi}\|J_2(x)(\xi)\|\\&=\|J_2(x)\|,
\end{align*}
so $J_2$ is an isometry. Since $J_2$ is unital, it is a Jordan isomorphism by \cite[Corollary~2.2]{LRW2}. Hence $J_2\oplus\tilde{J}\colon M\to N$ is a Jordan isomorphism that extends $\theta$.
\end{proof}

\footnotesize
\bibliographystyle{alpha}
\bibliography{hilbert_jb_bib}

\end{document}